\documentclass{amsart}

\usepackage{verbatim}

\usepackage[utf8]{inputenc}      
\usepackage[T1]{fontenc}         
\usepackage[american]{babel}

\usepackage{amssymb,amsmath,amscd,graphicx,color,enumerate, verbatim, latexsym}

\usepackage[%
   colorlinks,
   citecolor=red, linkcolor=blue,
   pdfpagelabels=true,
   unicode=true,
   pdfusetitle
 ]{hyperref}

\usepackage[normalem]{ulem}
 \renewcommand\sout{\bgroup\markoverwith
 {\textcolor{red}{\rule[0.7ex]{3pt}{1.4pt}}}\ULon}

\usepackage[font=bf]{caption}

\usepackage{tikz-cd}

 \usepackage[all]{xy}

 \usepackage{xcolor}                                           %
 \newcommand\magenta[1]{\textcolor{magenta}{#1}}               %
\newcommand{\clop}{\overline{\psi^{0}}(M; E)}
\newcommand{\clopn}{\overline{\psi^{-1}}(M; E)}

\newcommand{\Hom}{\operatorname{Hom}}
\newcommand{\End}{\operatorname{End}}

\newcommand{\CI}{\mathcal{C}^{\infty}}
\newcommand{\CIc}{\mathcal{C}_c^{\infty}}

\newcommand{\Ind}{\operatorname{Ind}}
\newcommand{\Prim}{\operatorname{Prim}}

\newcommand{\oid}{\operatorname{Id}}

\newcommand\bijchi{\magenta{\chi}}
\newcommand\one{\mathbf{1}}

\newcommand{\CC}{\mathbb C}

\newcommand{\RR}{\mathbb R}

\newcommand{\maC}{\mathcal C}

\newcommand{\maR}{\mathcal R}

\newcommand\ede{\, := \,}
\newcommand\seq{\, = \,}

\newtheorem{theorem}{Theorem}[section]
\newtheorem{lemma}[theorem]{Lemma}
\newtheorem{proposition}[theorem]{Proposition}

\newtheorem{problem}{Problem}

\theoremstyle{definition}
\newtheorem{definition}[theorem]{Definition}

\begin{document}

\title[Fredholm conditions]{Fredholm conditions for operators
invariant with respect to compact Lie group actions}
  
\author[A. Baldare]{Alexandre Baldare}
\email{alexandre.baldare@math.uni-hannover.de}
\address{Institut fur Analysis, Welfengarten 1, 30167 Hannover,
Germany}

\author[R. C\^ome]{R\'emi C\^ome} \email{remi.come@univ-lorraine.fr}
\address{Universit\'{e} Lorraine, 57000 Metz, France}

\author[V. Nistor]{Victor Nistor} \email{nistor@univ-lorraine.fr}
\address{Universit\'{e} Lorraine, 57000 Metz, France}
  \urladdr{http://www.iecl.univ-lorraine.fr/~Victor.Nistor}

\thanks{A.B., R.C., and V.N. have been partially supported by
  ANR-14-CE25-0012-01 (SINGSTAR). Manuscripts available from {\bf
    http:{\scriptsize//}www.math.psu.edu{\scriptsize/}nistor{\scriptsize/}.
} }


\begin{abstract}
Let $G$ be a compact Lie group acting smoothly on a smooth, compact
manifold $M$, let $P \in \psi^m(M; E_0, E_1)$ be a $G$--invariant,
classical pseudodifferential operator acting between sections of two
vector bundles $E_i \to M$, $i = 0,1$, and let $\alpha$ be an
irreducible representation of the group $G$. Then $P$ induces a map
$\pi_\alpha(P) : H^s(M; E_0)_\alpha \to H^{s-m}(M; E_1)_\alpha$
between the $\alpha$-isotypical components. We prove that the map $\pi_\alpha(P)$ is
Fredholm if, and only if, $P$ is {\em transversally $\alpha$-elliptic}, a condition
defined in terms of the principal symbol of $P$ and the action of
$G$ on the vector bundles $E_i$. 
\end{abstract}

\maketitle

\section{Introduction}

Throughout this paper, let $G$ be a compact Lie group acting smoothly 
and isometrically on a smooth, compact riemannian 
manifold $M$, let $P \in \psi^m(M; E_0, E_1)$ be a $G$--invariant,
classical pseudodifferential operator acting between sections of two
vector bundles $E_i \to M$, $i = 0,1$, and let $\alpha$ be an
irreducible representation of the group $G$. Then $P$ induces, by $G$-invariance, a map
\begin{equation}\label{eq.def.Pchi}
   \pi_\alpha(P) : H^s(M; E_0)_\alpha \to H^{s-m}(M; E_1)_\alpha
\end{equation}
between the $\alpha$-isotypical components of the corresponding
Sobolev spaces of sections. In this short note, we obtain necessary and sufficient
conditions for $\pi_\alpha(P)$ to be Fredholm. Fredholm operators are important in 
many applications to PDEs and geometry. Our result generalizes to the case of 
compact Lie groups the results of \cite{BCLN1,BCLN2},
which dealt with finite groups. We assume the reader is familiar with \cite{BCLN2}
and here we just explain the main differences from the case of finite groups. 

In order to state our result, we need to set up some notation, some already introduced 
in \cite{BCLN1, BCLN2}. As usual, $\widehat{G}$ denotes the set of equivalence classes 
of irreducible $G$-modules (or representations, 
agreeing that all our representations are strongly continuous).  
In general, if  $T : V_0 \to V_1$ is a $G$-equivariant linear map of $G$-modules and 
$\alpha \in \widehat G$, we let $\pi_\alpha(T): V_{0\alpha} \to V_{1\alpha}$ denote 
the induced $G$-linear map between  the $\alpha$-isotypical components of the $G$-modules
$V_i$, $i = 0, 1$. Since $P$ is $G$-invariant, its principal symbol $\sigma_m(P)$  belongs
 to $\in \maC^{\infty}(T^*M \smallsetminus \{0\}; \Hom(E_0, E_1))^G$. 
Let $G_\xi$  and $G_x$ denote the isotropy subgroups of $\xi \in
T_x^*M$ and $x \in M$, as usual. Then $G_\xi \subset G_x$ acts linearly
on the fibers $E_{0x}$ and on $E_{1x}$. Let $\mathfrak{g}$ denote the Lie algebra of $G$. 
Then any $Y \in \mathfrak{g}$ defines
a canonical vector field $Y_M$ on $M$. Let then as
in \cite{atiyahGelliptic} the {\em $G$-transverse cotangent space} be given by
\begin{equation}\label{eq.transverse.ts}
    T^*_G M \ede \{\, \xi \in T^*M\ | \ \xi(Y_M(\pi(\xi))) =
    0,\ \forall Y \in \mathfrak{g} \, \}.
\end{equation}
As before, $S^*M$ denotes the unit cosphere bundle of $M$.
Let $S^*_GM := S^*M \cap
  T^*_GM $ denote the set of unit covectors in the $G$-transverse cotangent
  space $T^*_GM$. Let then
\begin{equation}\label{eq.def.Gamma.symb} 
  \begin{gathered}
    \sigma_m^G(P)\ :\ \Omega_{M}(E) \ede  \{ ( \xi, \rho) \in S_G^*M \times \widehat G_\xi 
    \, \vert \ E_{\xi \rho} \neq 0 \}\ 
  \rightarrow \bigcup_{(x, \rho) \in \Omega_{M}(E)} \Hom(E_{0x
      \rho}, E_{1x \rho})^{G_\xi}\,, \\
    \sigma_m^G (P) (\xi, \rho) \ede \pi_{\rho}(\sigma_m(P)(\xi)) \in
    \Hom(E_{0x \rho}, E_{1x \rho})^{G_\xi}\,, \ \ \xi \in  (S_G^*M)_x  \,,
   \end{gathered}
\end{equation}
called the {\em $G$-principal symbol} $\sigma_m^G (P)$ of $P$. 

Recall also the following. If $A$ and $B$ are compact groups, $H$ is a subgroup of both $A$
and $B$, and $\mathrm{Hom}_H(\alpha, \beta)\neq0$, then $\alpha \in \hat{A}$ and $\beta \in \hat{B}$ 
are said {\em $H$-associated.} The group $G$ acts on $\{G_\xi \mid \xi \in T^*M\}$ by 
$g\cdot G_\xi :=G_{g\xi}=gG_\xi g^{-1}$.
For $\rho \in \widehat{G}_\xi$ define $g \cdot \rho \in
\widehat{G}_{g\xi}$ by $(g\cdot \rho)(h)=\rho(g^{-1}hg)$, for all
$h\in G_{g\xi }$. The characterization of Fredholm operators can be reduced to each
component of the orbit space $M/G$, and therefore we can and will assume $M/G$ to be \underline{connected}.
Under this hypothesis there exists a minimal isotropy subgroup $K$
such that any isotropy subgroup contains a subgroup conjugated with $K$, 
Furthermore, the set of points $M_{(K)}$ with stabilizer conjugated with
$K$ is an open dense submanifold of $M$ called the principal orbit bundle of $M$, see 
\cite{tomDieckTransBook}. Let us fix a minimal isotropy group $K \subset G$ for $M$ and
\begin{equation}\label{eq.def.Xalpha}
   \Omega^{\alpha}_{M} \ede \{(\xi, \rho) \in \Omega_{M}(E) \mid
   \text{$\exists g\in G$, $g\cdot \rho$ and $\alpha$ are
     $K$-associated} \}.
\end{equation}
 In the definition of the space $\Omega_M^\alpha$ above $g
  \in G$ is such that $K \subset g\cdot G_\xi$.

\begin{definition}\label{def.chi.ps}
The {\em $\alpha$-principal symbol} $\sigma_m^\alpha (P)$ of $P$ is
$\sigma_m^\alpha(P) \ede \sigma_m^G(P)\vert_{\Omega^\alpha_{M}}$.
We shall say that $P \in \psi^m(M; E_0, E_1)^{G}$ is {\em transversally 
  $\alpha$-elliptic} if its $\alpha$-principal symbol
$\sigma_m^\alpha(P)$ is invertible everywhere on its domain of
definition. 
\end{definition}

The  transversal $\one$-ellipticity is related with transversal ellipticity on (singular) 
foliations \cite{AnSk11a, connesBook, DLR}.
We can now formulate our main result.

\begin{theorem}\label{thm.main1}
Let $m\in \RR$, $P \in \psi^m(M; E_0, E_1)^{G}$ and $\alpha \in \widehat{G}$. Then
\begin{equation*}
    \pi_\alpha(P) : H^s(M; E_0)_\alpha \, \to \, H^{s-m}(M;
    E_1)_\alpha
\end{equation*}
is Fredholm if, and only if $P$ is transversally $\alpha$-elliptic.
\end{theorem}

For $G$ finite, our main result was proved before
\cite{BCLN1, BCLN2}. This is the first paper that deals with the
non-discrete case. As far as the statement of the result goes, the
case non-discrete is different from the discrete case in that $S_G^* M
\neq S^*M$. The proof in the non-discrete case, is, however,
significantly different from the one in the discrete case. Our
results are motivated, in part, by questions in Index Theory and also
by the recent improvement \cite{bruning2019some} and the reference therein.
The techniques used in this paper to obtain Fredholm conditions
are related also to the ones in
\cite{SavinSchrohe}, used for $G$-opeators, and the ones in \cite{BL92},
used for complexes of operators. See also \cite{DLR, MSS, vEY2}.

We thank Matthias Lesch, Paul-Emile Paradan, and Elmar Schrohe for useful 
discussions.

\section{Background material}
\label{sec2}

This section is devoted to background material and results. The reader can find
more details in \cite{BCLN1, BCLN2}. We concentrate only on the material that
is significantly different from the discrete case, for which we refer to 
\cite{BCLN2}. There is no loss of generality to assume that $M/G$ is
connected (recall that $G$ is a compact Lie group acting by isometries on a
compact Riemannian manifold $M$). A $G$-module will be a strongly continuous 
representation of $G$.

One of the main differences in the non-discrete case is that we need two versions 
of induction. Let $H \subset G$ be a closed subgroup and $V$ be an $H$-module, 
we define, as usual, the continuous
induced representation by
\begin{equation}\label{eq.def.induced}
  \begin{split}
    c_0\mbox{-}\Ind_H^G (V)  \ede \ & \maC(G, V)^H 
  =\ 
  \{\, f \in \maC( G, V) \, |\  f(gh^{-1}) = h f(g) \, \} .
  \end{split}
\end{equation}
Assume that $V$ is a Hilbert space and that the representation is
unitary. Then we let the {\em
  hilbertian induced representation} $L^2\mbox{-}\Ind_H^G(V)$ be
the completion of $c_0\mbox{-}\Ind^G_H(V)$ with respect to the induced norm.
Then $G$ acts by left translation on $c_0\mbox{-}\Ind_H^G (V)$ and
$L^2\mbox{-}\Ind_H^G(V)$. Our proofs
use Frobenius reciprocity for both types of induction.

Let us summarize two important properties of the $G$-transversal spaces $T_G^*M$
and $S_G^*M$.

\begin{lemma}\label{lemma.comp.transverse} 
Let $H$ be a closed subgroup of $G$ and $S$ be a $H$-manifold. Then
$T^*_G(G\times_H S) \cong G\times_H T^*_HS\,.$ Moreover, 
The subset $S^*_G M_{(K)}$ is dense in $S^*_G M$.
\end{lemma}

Let $A_{M} \ede \maC(S^*_G M; \End(E))$. Then $\Prim(A_M^G)$, the primitive ideal spectrum of 
$A_M^G$, identifies with the set $\Omega_{M}(E)/G$. See also \cite{EchterhoffWilliams14}.
Explicitly, for any $\sigma \in A_M^G$ and $(\xi, \rho) \in \Omega_M^G(E)$
(see Equation \eqref{eq.def.Gamma.symb}), we define
\begin{equation}\label{eq.pi.xi}
    \pi_{(\xi,\rho)}(\sigma) \ede \pi_\rho(\sigma(\xi))=\sigma(\xi )\vert_{E_{\xi\rho}}.
\end{equation}
Then the map $\bijchi : \Omega_M(E) \to \Prim(A_M^G)$ given by $\bijchi(\xi, \rho) = 
\ker \pi_{(\xi, \rho)}$ induces a bijection $\bijchi_0 : \Omega_M(E)/G \to \Prim(A_M^G)$.

\section{Primitive ideals and the proof of the main theorem}

Let $\overline{\psi^0}(M,E)$ (respectively $\overline{\psi^{-1}}(M,E)$) be the norm closure of 
the algebra of compactly supported
classical, order zero (respectively, order $-1$) pseudodifferential operators on $M$.
Let $K$ be our fixed minimal isotropy group, let 
$\alpha\in \widehat{G}$, and let $\pi_\alpha$ the restriction morphism to the
$\alpha$-isotypical component $L^2(M;E)_\alpha$ of $L^2(M;E)$, see Equation \eqref{eq.def.Pchi}.
As for the discrete case, the most important (and technically difficult)
part is the identification of the quotient $\pi_\alpha(\clop^ G)/\pi_\alpha(\clopn^ G)$.

As in \cite{atiyahGelliptic}, the map $\maC(S^*M; \End(E))^G 
   \rightarrow \pi_\alpha( \clop^G )/\pi_\alpha( \clopn^G )$ 
descends to a surjective map 
\begin{equation}
\widetilde \maR^{\alpha}_{M} : 
A_M^G \ede \maC(S_G^*M; \End(E))^G \to \pi_\alpha(\overline{\psi^0}(M,E)^G)/\pi_\alpha(\overline{\psi^{-1}}(M,E)^G).
\end{equation}
Since the  map $\widetilde
\maR^{\alpha}_{M}$ is surjective, the question of determining the
quotient algebra $\pi_\alpha(\overline{\psi^0}(M,E)^G/
\pi_\alpha(\overline{\psi^{-1}}(M; E)^G)$ is equivalent to the
question of determining the ideal $\ker (\widetilde \maR_{M}^\alpha)
\subset A_M^ G$. In turn, this ideal will be determined by solving the
following problem:

\begin{problem}\label{problem1}
Let $A_M^G := \maC(S^*_G M; \End(E))^G$, as before. Identify the
closed subset
\begin{equation}\label{eq.def.Xi}
   \Xi^\alpha(E) \ede \Prim(A_M^ G/ \ker(\widetilde \maR_{M}^\alpha))
   \, \subset \, \Prim(A_M^ G) \,.
\end{equation}
\end{problem}

\subsection{Calculation on the principal orbit bundle}

By tensoring with $\alpha^*$, we can assume that $\alpha = \one$, the
trivial representation.
Let $x_0\in M_{(K)}$ with $G_{x_0}=K$, 
let $U \subset (T^*_GM)_{x_0}$ be a slice at $x_0$, let and $W=G\exp_{x_0}(U) \cong G/K \times U$ 
be a tube around $x_0$. Let
\begin{equation}\label{eq.def.notation}
    0 \neq \eta \in E_{x_0}^{K}\,,\ \xi \in (S^*_{K} U)_{x_0} \seq
    S^*_{x_0} U\,,\ \mbox{ and } \ f \in \CIc(U)\,, \ f(x_0) = 1 \,,
\end{equation}
where $S_K^*U = S^*U$, since $K$ acts trivially on $U$ by minimality. 
We define then $s_\eta \in \CIc(W; E)^G$ and $e_t \in \maC(W)^G$ by
 $s_\eta( g \exp_{x_0} (y) ) \ede f(y) g\eta$ and 
$  e_t( g \exp_{x_0} (y))
  \ede e^{\imath t \langle y, \xi \rangle}$
that is, they are the $G$-invariant functions
extending the functions $y \mapsto f(y) \eta$ and $y \mapsto
e^{\imath t \langle y, \xi \rangle}$ by $G$-invariance via $W = G
\exp_{x_0}(U)$, where $y \in U \subset T_{x_0}U$ and $t \in \RR$.
Let us notice that, if we let $\Phi_K^G$ denote the Frobenius isomorphism, then
$s_\eta \ede \Phi_K^G(f\eta)$ and $   e_t \seq \Phi_K^G(e^{\imath t \langle \bullet , \xi \rangle})$. 
Using oscillatory testing techniques, see, for instance \cite{Hormander3, Treves}, we obtain.

\begin{proposition}\label{prop.injective}
Assume that $E_{x_0}^K \neq0$. Then, for every $P \in \psi^0(M; E)$, we have
$\lim_{t \to \infty } P ( e_t s_\eta ) (x_0) \seq
  \sigma_0(P)(\xi)\eta$.
In particular, if $P\in\psi^0(M; E)^G$, then
\begin{equation*}
  \lim_{t \to \infty } \pi_{\one}(P) ( e_t s_\eta ) (x_0) \seq
  \sigma_0(P)(\xi)\eta =: \, \pi_{(\xi, \one)} \big (\sigma_0(P)\big )
  \eta\,.
\end{equation*}
Moreover, $\pi_{(\xi, \one)} \in \Xi^{\one}(E)$. Equivalently,
$\bijchi \big ( \Omega^{\one}_{M_{(K)}}(E) \big ) \subset \Xi^{\one}(E)$.
\end{proposition}

Above $M_{(K)}$, we also have the opposite inclusion.

\begin{theorem}\label{theorem.princ.str}
We have $\bijchi_0 ( \Omega^{\one}_{M_{(K)}}(E)/G) =
\Xi^{\one}(E) \cap \Prim(A_{M_{(K)}}^G)$. In other words,
$$ \Xi^{\one}_0(E) \ede \Xi^{\one}(E) \cap \Prim(A_{M_{(K)}}^G) \seq
   \{\ker \big (\pi_{(\xi, \one_{G_\xi})} \big )\, \vert \ \xi \in
   S_G^*M_{(K)} ,\, E_\xi^{G_\xi}\neq 0 \}.$$
\end{theorem}

\begin{proof}
Notice that $\{\ker \big (\pi_{(\xi, \one_{G_\xi})} \big )\, \vert \ \xi \in
   S_G^*M_{(K)} ,\, E_\xi^{G_\xi}\neq 0 \}\subset\bijchi_0 ( \Omega^{\one}_{M_{(K)}}(E)/G)$.
   Reciprocally, by conjugation, we can assume that $\ker \big (\pi_{(\xi, \rho)} \big )
   \in \bijchi_0 ( \Omega^{\one}_{M_{(K)}}(E)/G)$ is such that $G_\xi=K$ and then $\Hom_K(\rho , \one_G)\neq 0$ is equivalent
   to $\rho=\one_K \in \hat{K}$.
Proposition \ref{prop.injective} states that $\bijchi
( \Omega^{\one}_{M_{(K)}}(E) ) \subset \Xi^{\one}(E) $, and hence 
$\bijchi ( \Omega^{\one}_{M_{(K)}}(E)) \subset \Xi^{\one}_0(E)$.
Let $\bijchi(\xi,\rho)=\ker(\pi_{(\xi,\rho)}) \in \Xi^\one_0(E)$, by definition this means:
\begin{itemize}
\item $(\xi, \rho) \in \Omega_{M}(E)$ and $\xi \in S^*_G M_{(K)}$;

\item $\pi_{(\xi, \rho)}$ is as defined in Equation
  \eqref{eq.pi.xi}; and, most importantly,

\item  $\pi_{(\xi,\rho)}  \vert_{\ker (\widetilde \maR^\one_{M})} = 0$.
\end{itemize} 
We need to prove that $(\xi, \rho) \in \Omega_{M_{(K)}}^{\one}(E)$. 
As before, we can assume that $G_\xi = K$ and therefore we only need to prove that 
$E_\xi^K \neq 0$ and  $ \rho = \one_K$.
Since $\pi(\xi) := x_0 \in M_{(K)}$, we can replace $M$ with the tube $W$ as before.
We shall prove that $E_\xi^K \neq 0$ by
contradiction. Indeed, if $E_{x_0}^K = 0$, then $L^2(W; E)^G = 0$, by
induction, and hence $\ker(\widetilde \maR^\one_{W}) =
A_{W}^G$. But this is not possible since  $\ker(\widetilde\maR^\one_{W}) \subset \ker(\pi_{(\xi,\rho)})$ 
and $A^G_W$ is not a primitive ideal.
  
We shall prove that $\rho = \one_K$ by contradiction, as well.
Assume hence that $\one_K \neq \rho \in \widehat{K}$.  Let $p_{\rho}$ be the projection
onto the isotypical component corresponding to $\rho$ in
$\End(E_\xi)^K \simeq \End(E_x)^K$. We have $p_\rho \neq 0$ since
$(\xi , \rho) \in \Omega_{M}(E)$. Let $f \in \CIc(U)$ be equal to $1$
near $x_0$ and extend $f p_{\rho}$ to a $G$--invariant element
$\Phi_K^G(f p_{\rho}) \in \CI(M; \End(E))^{G} \subset \overline{\psi^0}(M,E)^G$ via $W =
G\exp(U)$. We have
\begin{equation*}
  \pi_{(\xi,\rho)}(\Phi_K^G(f p_{\rho})) \seq p_{\rho} \quad \mbox{and} \quad 
  \pi_{\one}(\Phi_K^G(fp_\rho)) =
\Phi_K^G(fp_\rho)p_{\one_G} = 0  \,,
\end{equation*}
by construction, because $p_{\one_G}(y)=p_{\one_K}$ for any $y\in U$ and $\rho\neq \one_K$. 
 Therefore,
$\Phi_K^G(fp_\rho)\in \ker( \widetilde \maR_{M}^{\one})$ and that
$\pi_{(\xi,\rho)}(\Phi_K^G(fp_\rho))\neq 0$, which contradicts our
assumption that $\pi_{(\xi,\rho)}\vert_{\ker( \widetilde
  \maR_{M}^{\one})} = 0$. Hence $\rho = \one_K$. 
\end{proof}

\subsection{Calculation for singular $(\xi,\rho)$}

\subsubsection{A particular neighborhood of $(\xi,\rho)$.}
Let $H$ be a closed subgroup of $G$ and $V$ a vector space on which $H$ acts by isometries. 
Let $W=G \times_H V$ and assume that $E=G \times_H (V\times E_0)$ for some $H$-module $E_0$.
Let $(\xi, \rho) \in \Omega_{W}(E)$ be our fixed element. Then here 
$S^*W=G \times_H (V \times S^*_xW)=G \times_H  (V \times S)$,
where $S \subset (\operatorname{Lie}(G)/\operatorname{Lie}(H))^* \times V^*$ is the unit sphere.
We assume that $x=\pi(\xi)=H(e,0)$. Thus $\xi \in
(S^*_ GW)_{x}=(G \times_H S^*_HV)_{x}=S^*_0V$ and $\rho \in \widehat{ G}_\xi$ is an irreducible
representation that appears in $E_0=E_\xi = E_x$ (by the definition
of $\Omega_{W}(E)$). We can assume that $K \subset G_\xi$. Let, as before, $p_\rho$ denote 
the projection onto the $\rho$-isotypical component, as before \ref{sec2}.
Let $F_\xi := \{\rho' \in \widehat{G}_\xi, \rho' \subset E_0\ \mbox{and } 
\rho^{'K}\neq 0\}$ and $p_K  \ede  \bigoplus_{\rho' \in F_\xi} p_{\rho'} 
  \in \End(E_0)^{G_\xi}$ and $ p_{\xi,K} = \oid_{E_0} - p_K$.
  The
extension $\widetilde q$ to the sphere $S_x^*W$ will not be constant. We thus
define first $p \in \maC_c^\infty(S^*_{x} W; \End(E))^{G_\xi} \simeq
  \maC_c^\infty(S^*_{x} W; \End(E_0))^{G_\xi}$
  such that $p(\xi) = p_{\xi,K}$. Let 
$\overline{U'_\xi} \subset U_\xi \, \subset \,  T_\xi (S_{x}^*W)\cap
      \operatorname{Lie} (H)^{\perp} \, =: \, (T_{H} S_{x}^*W
    )_\xi$
be slices at $\xi$ for
the action of $H$ on the sphere $S_{x}^*W$.  Let $f\in
\maC_c^\infty(U_\xi)$ be $\ge 0$ and equal to $1$ on
$U'_\xi$. We then define $ p \in \maC^\infty(S^*_{x} W; \End(E))^{H} \, \simeq\,
     \maC^\infty(S^*_{x} W; \End(E_0))^{H}$ by 
\begin{equation}\label{eq.tilde.p}
  \begin{gathered}
     p(g \exp(\eta)) \ede
\begin{cases}
     \, f(\eta) g p_{\xi,K} g^{-1} \,,\ & \mbox{ for } \ \eta \in
     U_\xi \,, \ g \in H \\
     \quad 0 \,, &  \mbox{ {\em outside} the tube }\ H
     \exp_\xi(U_\xi) \,.
     \end{cases}
  \end{gathered}
\end{equation} 
Using the fact that $E$ and $S^*W$ are trivial on $V$, 
we then extend $p$ to an element $\widetilde
p \in \maC^\infty(S^* W\vert_{V}; \End(E))^{H} \simeq
\maC^\infty(S_x^*W \times V; \End(E_0))^{H}$ {\em that is
  independent of the coordinate $V$} and then we finally further
extend this element by $G$--invariance to an element
$q$ in $\maC^\infty(S^*W; \End(E))^{G}$, i.e.,
\begin{equation*}\label{eq.tilde.q}
     q \ede \Phi_{H}^G(\widetilde p) \in \maC^\infty(S^*
     W; \End(E))^{G}\,,\ \mbox{ where }\ \widetilde p(y,
     \zeta) \ede p(\zeta), \ \zeta \in S_x^*W\,.
\end{equation*}
Finally, let $h \in \CI_c(W,[0,1])^G$ be a function equal to $1$ on a $G$-invariant neighborhood 
 of $(x,\xi)$ and let $\widetilde q=hq \in \CI_c(S^*W,\End(E))^G$. 

Recall that $p_{\one_G}$ is defined for any $s \in \CI(W,\End(E))$ by 
$(p_{\one_G}s)(z)=\int_G g(s(g^{-1}z)) dg,$
and that above $W_{(K)}$ this defines a bundle morphism $\mathbb{P}:=p_{\one_G}\vert_{W_{(K)}} \in \CI(W_{(K)},\End(E))$.
More precisely, we have $\mathbb{P}(y) = p_{\one_{G_y}}$, for $y \in W_{(K)}$. 
Then $\mathbb{P} \in \CI(W_{(K)}; \End(E))^G$ is a smooth map of projections.

\begin{lemma}\label{lemma.q}
Let $\bijchi(\xi,\rho)\in \Omega_W(E) $ with $K \subset G_\xi$. Assume that $\rho^K=0$. Then
\begin{enumerate}[(i)]
\item $\pi_{(\xi,\rho)}(\widetilde q) = p_\rho \neq 0$.
\item $\pi_{(\zeta,\one_{G_\zeta})}(\widetilde q)=0$, for all $\zeta \in S^*M_{(K)}$.
\item $\mathbb{P} \widetilde q\vert_{S^*M_{(K)}} = 0$.
\item $V_{\widetilde q} = \{(\xi,\rho)\in \Omega_M(E) \mid \pi_{(\xi,\rho)}(q)\neq 0\}$ is an open 
neighbourhood of $\bijchi(\xi,\rho)$ and $V_{\widetilde{q}} \cap \Xi^\one_0(E)=\emptyset$. 
In particular, $\bijchi(\xi,\rho)\notin \overline{\Xi^\one_0(E)}  $.
\end{enumerate}
\end{lemma}

\begin{proof}
Let us notice first that
$p_\rho \in \End(E_\xi)^{G_\xi}=\End(E_0)^{G_\xi}$ is non zero because $\bijchi(\xi,\rho)\in \Omega_W(E)$ 
implies that $\rho \subset E_\xi$, by definition.
\
(i)\ Since $\rho^K=0$ we get that $\rho \notin F_\xi$. Therefore, for any 
$\rho'\in F_\xi$, $p_{\rho'}p_\rho=0$. 
This gives that $\pi_{(\xi,\rho)}(\widetilde q) 
= p_\rho -\bigoplus_{\rho'\in F_\xi} p_{\rho'} p_{\rho} = p_{\rho} 
= Id_{E_{\xi\rho }} \in \End(E_{\xi\rho })$. \
(ii)\
In view of the support of $\widetilde{q}$ and its $G$-invariance, this reduces 
to $p_{\xi,K} p_{\one_K}=p_{\one_K}p_{\xi,K}=0$
because $E_\xi^K=\bigoplus_{\rho'\in F_\xi} E_{\xi\rho'}^K$ and therefore 
$p_K p_{\one_K}=p_{\one_K}$.\ (iii)\ This is just another formulation of (ii). 
\ (iv)\ It is standard that the set $V_{\widetilde{q}}$ is open, see for example \cite{BCLN2}. By 
Theorem \ref{theorem.princ.str} $\Xi^\one_0(E)=\{\ker(\pi_{(\zeta,\one_{G_\zeta})}), 
\zeta \in S^*_GW_{(K)},\ E_{\zeta}^{G_\zeta}\neq 0\}$ and then $(2)$ implies $(3)$. 
\end{proof}

\subsubsection{Density of {$\Xi_0^{\one}$} in {$\Xi^{\one}$}}

We freely use the notations from the previous section.

\begin{lemma}\label{lem.proof.main.thm}
Assume that $(\xi,\rho) \notin \overline{\Xi^\one_0(E)}$ and that $K\subset G_\xi$. Then 
\begin{enumerate}[(i)]
\item $\rho^K=0$, in particular $\Omega^\one_M(E)/G=\overline{\Xi^\one_0(E)}
=\overline{\Omega^\one_{M_{(K)}}(E)/G}$ and
\item there is $\widetilde{Q} \in B_W^G$ such that $\sigma_0(\widetilde{Q})
=\widetilde q$ and $\widetilde{Q} \vert_{L^2(W,E)^G}=0$.
\end{enumerate}

\end{lemma}

\begin{proof}
(i)\ We already showed in Theorem \ref{theorem.princ.str} that $\Xi^\one_0(E)=\Omega^\one_{W_{(K)}}(E)/G$.
Since $\ker(\pi_{(\xi,\rho)})$ is maximal 
($A^G_M/\ker(\pi_{(\xi,\rho)})=\End(E_{\xi\rho})^{G_\xi}=M_k(\CC)$ since $\rho \in \widehat{G}_\xi$),
there is $\sigma \in A^G_M$ such that $\pi_{(\xi,\rho)}(\sigma)=\oid$ 
and the associated neighbourhood $V_\sigma = \{(\eta,\rho'), \pi_{(\eta,\rho')}(\sigma)\neq 0\}$ 
does not intersect $\overline{\Xi^\one_0(E)}=\overline{\Omega_{M_{(K)}}^\one(E)}$. 
Assume $\rho^K\neq 0$ 
and let $\zeta_n \in S^*_GW_{(K)}$ be a convergent sequence to $\xi$ using Lemma \ref{lemma.comp.transverse} and write
$\sigma(\zeta_n)\vert_{E_{\xi\rho}}=\sigma(\zeta_n)\vert_{E_{\xi\rho}^K}\oplus \sigma(\zeta_n)\vert_{(E_{\xi\rho}^K)^\perp}$.
But $E_{\xi\rho}^K \subset E_{\xi}^K$ and $\sigma(\zeta_n)\vert_{E_{\xi}^K}=\pi_{(\zeta_n,\one)}(\sigma)=0$. Therefore,
$\pi_{(\xi,\rho)}(\sigma)=\lim \sigma(\zeta_n)\vert_{E_{\xi\rho}}$ is not invertible, a contradiction.\\

(ii)\
$\bullet$ \underline{\emph{Quantization of $\widetilde{q}$.}}
Cover $W=G \times_H V \to G/H$ by a finite number of local charts $Y_i \simeq D_i \times V$ centred in $y_i \in Y_i$,
 where
$D_i \subset G/H$ is a local chart. Then $E$ is trivial over $Y_i$ and we have $E\vert_{Y_i} \simeq D_i \times V \times E_0$.
Let $(\varphi_i^2)$ be a subordinated partition of unity to $D_i$ on $G/H$ (and then on $W$). 
Let $\chi_i \in \maC_c^\infty(Y_i)$ be such that $\chi_i
\varphi_i = \varphi_i$.
Then let $\psi_i \in \maC^\infty(T^*_{y_i}Y_i)$ be
such that $\psi(0) = 0$ if $|\eta| < 1/2$ and $\psi(\eta) = 1$
whenever $|\eta| \ge 1$.  Let for any classical symbol $a$ on $Y_i$, 
 and $s\in \maC^\infty_c(Y_i,E)$ 
\begin{equation*}
  Op(a) s(y) \ede \int_{T^*_{y_i}Y_i}\int_{Y_i} e^{\imath(y-z)\cdot \eta}\, a(y,
  z, \eta) s(z) dz d\eta ,
\end{equation*}
where we use the normalized measure $d\eta=(2\pi)^{-\dim W}\ d\eta_1 \cdots d\eta_{\dim W}  $.
We shall use this for $a_i(y,z, \eta) := \chi_i(y) \psi_i(\eta) \widetilde
q\Big(y,\frac{\eta}{|\eta|}\Big)\chi_i(z)$, then set 
$Q_i \ede Op(a_{i})$. 
Now define $Q:= \sum_i \varphi_i Q_i \varphi_i$ and 
$\widetilde{Q} : = \operatorname{Av}(Q)= \int_G gQg^{-1} dg.$
Then $\widetilde{Q}$ is a zero order $G$-invariant pseudodifferential
operator on $W$ representing $\widetilde q$ because the average map commutes with the principal symbol map. 
\bigskip

$\bullet$ \underline{\emph{Proof of $\widetilde{Q}\vert_{L^2(W,E)^G}=0$.}} 
By density, it is sufficient to show that 
$\widetilde{Q}s(y)=0 $ for  $y\in W_{(K)}$ and $s\in \maC^\infty_c(W,E)^G.$
By Lemma \ref{lemma.q}(iii), we have for $y\in W_{(K)}\cap Y_i$, $z\in Y_i$, and $\eta \in T^*_{y_i}Y_i$
that $\mathbb{P}a_i(y,z,\eta)=0$, and therefore
$\mathbb{P}a_i$ extends in a function $\widetilde{\mathbb{P}a_i}$ identically zero on $W$.

Thus for any $s\in \CI_c(Y_i,E)$ and $y\in Y_i \cap W_{(K)}$, we obtain that
$\mathbb{P}(y)Q_i(s)(y)=0$.
Thus we get that $\mathbb{P}Qs(y)=0$, for any $y\in W_{(K)}$ and 
this in turn gives that $\mathbb{P}Q$ extends to $0$ on $L^2(W,E)$. Using the average map, we obtain
\begin{equation}\label{eq.PQ=0}
0=\operatorname{Av}(\mathbb{P}Q)s(y)=\mathbb{P}(y)\operatorname{Av}(Q)s(y)=\mathbb{P}(y)\widetilde{Q}s(y), \quad \forall y\in W_{(K)}.
\end{equation}

It follows that $\mathbb{P}\widetilde{Q}$ extends to $0$ on $L^2(W,E)$.
Now, assume that $s\in \CI_c(W,E)^G$ then $\widetilde{Q}s \in \CI_c(W,E)^G$ 
and therefore by Equation \eqref{eq.PQ=0} we have 
\begin{equation*}
   (\widetilde{Q}s)(y) \seq \mathbb{P}(y)(\widetilde{Q}s)(y)=0.
\end{equation*}

This implies that $\widetilde Q \vert_{L^2(W,E)^G} = 0$ and this completes the proof.

\end{proof}

\begin{theorem}\label{theorem.non.principal} 
The set $\Xi^{\one}(E) := \Prim(A_M^ G/\ker(\widetilde \maR_M^{\one}
)) \subset \Prim(A_M^ G)$ associated to the ideal $\ker
(\widetilde{\maR}_M^{\one})$ of $A_M^G := \maC(S^*_G M; \End(E))^G$ is
the closure in $\Prim(A_M^ G)$ of the set $\Xi^{\one}_0(E) :=
\Xi^{\one}(E) \cap \Prim(A_{M_{(K)}}^ G)$, where $M_{(K)}$ is the principal
orbit bundle of $M$. 
\end{theorem}

\begin{proof}
Since $\Xi^\one(E)$ is closed, it is enough to show that if $\ker \pi_{(\xi,\rho)}\notin
\overline{\Xi^{\one}_0(E)}$ then $\ker \pi_{(\xi,\rho)}\notin
\Xi^{\one}(E)$. Replacing $\xi$
  with some $g \xi$, $g \in G$, and $M$ with a tube around $x=\pi(\xi)$,  we may assume that $K \subset
G_\xi$ and $M=G\times_H V$ as before.
From Lemma \ref{lem.proof.main.thm} (i), we knows that $\rho^K=0$.
Furthermore Lemma \ref{lem.proof.main.thm}(ii) gives that there is $\widetilde{Q} \in \psi^0(M,E)^G$ such that
$\sigma_0(\widetilde{Q})=\widetilde q$ 	and $\widetilde{Q}\vert_{L^2(M,E)^G}=0$.
Therefore, $\sigma_0(\widetilde{Q})=\widetilde q\in \ker \widetilde \maR_M^\one$.
Since, $\pi_{(\xi,\rho)}(\widetilde q) \neq 0$ by Lemma \ref{lemma.q} $(i)$, 
$\pi_{(\xi, \rho)}$ does not vanish on $\ker( \widetilde
\maR_{M}^{\one})$, which means that $ \ker(\pi_{(\xi,
  \rho)}) \notin \Xi^{\one}(E)$ and this completes the proof.
\end{proof}

\end{document}